\documentclass[12pt]{amsart}
\usepackage{amsbsy,amssymb,amscd,amsmath,amsthm, hyperref,tikz,subcaption, pgf}
\usepackage{a4wide}
\usetikzlibrary{arrows}

\newcommand{\out}[1]{\ }

\newcommand{\DD}{{\mathbb D}}

\newcommand{\NN}{{\mathbb N}}
\newcommand{\QQ}{{\mathbb Q}}
\newcommand{\RR}{{\mathbb R}}
\newcommand{\CC}{{\mathbb C}}

\renewcommand{\ge}{\geqslant}
\renewcommand{\le}{\leqslant}
\renewcommand{\phi}{\varphi}
\renewcommand{\epsilon}{\varepsilon}
\renewcommand{\Re}{\re}
\renewcommand{\Im}{\im}

\newcommand{\eps}{\varepsilon}

\DeclareMathOperator{\re}{Re}
\DeclareMathOperator{\im}{Im}

\DeclareMathOperator{\inte}{int}
\DeclareMathOperator{\cl}{cl}
\newcommand{\fint}{\inte_f\!}
\newcommand{\fcl}{\cl_f\!}

\makeatletter
\newtheorem*{rep@theorem}{\rep@title}
\newcommand{\newreptheorem}[2]{%
\newenvironment{rep#1}[1]{%
 \def\rep@title{#2 \ref{##1}}%
 \begin{rep@theorem}}%
 {\end{rep@theorem}}}
\makeatother

\newtheorem{theorem}{Theorem}[section]
\newreptheorem{theorem}{Theorem}

\newtheorem{lemma}[theorem]{Lemma}
\newtheorem{prop}[theorem]{Proposition}

\theoremstyle{definition}
\newtheorem{definition}[theorem]{Definition}
\newtheorem{example}[theorem]{Example}
\newtheorem{question}[theorem]{Question}

\theoremstyle{remark}
\newtheorem{remark}[theorem]{Remark}

\numberwithin{equation}{section}

\newtheoremstyle{case}
{3pt}
  {3pt}
  {}
  {}
  {\bfseries}
  {:}
  {.5em}
  {}
\theoremstyle{case}

\numberwithin{subcase}{case}

\begin{document}
\title{Domains of existence for finely holomorphic functions}
\author{Bent Fuglede}
\address{Department of Mathematical Sciences\\
University of Copenhagen
\\Universitetsparken 5 
\\2100 K\o benhavn, \\Denmark}
\email{fuglede@math.ku.dk}
\author{Alan Groot}

\author{Jan Wiegerinck}
\address{KdV Institute for Mathematics
\\University of Amsterdam
\\Science Park 105-107
\\P.O. box 94248, 1090 GE Amsterdam
\\The Netherlands}\email{alangroot@gmail.com}
\email{j.j.o.o.wiegerinck@uva.nl}
\thanks{We are grateful to Jan van Mill for many enlightening discussions.}
\begin{abstract} We show that fine domains in $\CC$ with the property that they are  Euclidean $F_\sigma$ and $G_\delta$, are in fact fine domains of existence for finely holomorphic functions. Moreover \emph{regular} fine domains are also fine domains of existence. Next we show that fine domains such as $\CC\setminus \QQ$ or $\CC\setminus (\QQ\times i\QQ)$, more specifically fine domains $V$ with the properties that their complement contains a non-empty polar set $E$ that is of the first Baire category in its Euclidean closure $K$ and that $(K\setminus E)\subset V$, are \emph{not} fine domains of existence.

\end{abstract}
\maketitle

\section*{Introduction}
It was already known to Weierstrass that every domain $\Omega$ in $\CC$ is a \emph{domain of existence}, roughly speaking, it admits a holomorphic function $f$ that cannot be extended analytically at any boundary point of $\Omega$. In his thesis \cite{Bo1892} Borel showed, however, that it may be that $f$ can be (uniquely) extended to a strictly larger set $X$ in an "analytic" way, albeit that $X$ is no longer Euclidean open. This eventually led Borel to the introduction of his Cauchy domains and monogenic functions, cf.~\cite{Bo1917}. Finely holomorphic functions on fine domains in $\CC$ as introduced by the first named author, are the natural extension and setting for  Borel's ideas, see \cite{Fu81}, also for some historic remarks on Borel's work. In this paper we will study \emph{fine domains of existence}, roughly speaking,  fine domains in $\CC$ that admit a finely holomorphic function that cannot be extended as a finely holomorphic function at any fine boundary point. Definition \ref{Def.existence} contains a precise definition. Here the results are different from both the classical Weierstrass case of one variable and the classical several variable case, where Hartogs showed that there exist domains $D\subsetneq D^*$ in $\CC^n$ with the property that every holomorphic function on $D$ extends to the larger domain $D^*$.

 Our results are as follows. In Section \ref{sec2} we show that every fine domain that is a Euclidean $F_\sigma$ as well as a Euclidean $G_\delta$ is a fine domain of existence. Euclidean domains are of this form and therefore are fine domains of existence. We also show that \emph{regular} fine domains are fine domains of existence. Noting that on Euclidean domains holomorphic and finely holomorphic functions are the same, our Theorem \ref{Ufinecompactexhaustion} includes Weierstrass' theorem. In Section \ref{sec3}, however, we show that fine domains $V$ with the property that their complement contains a non-empty polar set $E$ that is of the first Baire category \emph{in its Euclidean closure}  $K$ (in particular, $E$ has no Euclidean isolated points) and $(K\setminus E)\subset V$, are \emph{not} fine domains of existence. 

 Pyrih showed in \cite{pyrih} that the unit disc is a fine  domain of existence, and that as a corollary of the proof, the same holds for simply connected Euclidean domains by the Riemann mapping theorem. 
 
The starting point of  our research was  the following observation. Edlund \cite{Edl} showed that  every closed  set $F\subset \CC$ admits a continuous function  $f:F\to \CC$ such that the graph of $f$ is completely pluripolar in $\CC^2$. It is easy to see that by construction this $f$ is finely holomorphic on the fine interior of  $F$.  In \cite{EEW} the main result can be phrased as follows: \emph{If a finely holomorphic function on a fine domain $D$ admits a finely holomorphic extension to a strictly larger fine domain $D'$, then the graph of $f$ over $D$ is not completely pluripolar. } Hence,  if $D$ is the fine interior of a Euclidean closed set, every fine component of $D$ is a fine domain of existence. This is extended by our Theorem \ref{Ufinecompactexhaustion} because the fine interior of any finely closed set is regular. 

In the next section we recall relevant results about the fine topology  and fine holomorphy.
\section{Preliminaries on the fine topology and finely holomorphic functions} 

Recall that a set $E\subset \CC$ is  \emph{thin} at $a\in \CC$ if $a\notin \overline E$ or else if there exists a subharmonic function $u$ on an open neighborhood of $a$ such that
\[\limsup_{z\to a, z\in E\setminus \{a\}}u(z)<u(a).\]
The fine topology was introduced  by H. Cartan in a letter to Brelot as the weakest topology that makes all subharmonic functions continuous. He pointed out that $E$ is thin at $a$ if and only if $a$ is not in the fine closure of $E$. The fine topology has the following known features. Finite sets are the only compact sets in the fine topology, which follows easily from the fact that every polar set $X\subset\CC$ is discrete in the fine topology, it consists of finely isolated points. The fine topology is Hausdorff, completely regular, Baire, and quasi--Lindel\"of, i.e. a union of finely open sets equals the union of a countable subfamily and a polar set, cf.~e.g.~\cite[Lemma71.2]{ArGa} . For our purposes fine connectedness is important.

Recall the following 
\begin{theorem}[Fuglede]\label{locallyconnected}
The fine topology on $\CC$ is \emph{locally connected}. That is, for any $a \in \CC$ and any fine neighborhood $U$ of $a$, there exists a finely connected finely open neighborhood $V$ of $a$ with $V \subset U$.
\end{theorem}

See \cite[p.92]{Fu} (or see \cite{ElMarzguioui2006} for a proof using only elementary properties of subharmonic functions).

In fact, the fine topology is even  \emph{locally polygonally arcwise connected} as was shown in \cite{Fu}, but we need something stronger.
\begin{definition}
 A \emph{wedge} is  a polygonal path consisting of two line segments $[a,w]$ and $[w,b]$ of equal length.  
\end{definition}

 The result we need is as follows, 
 
 \begin{theorem}[\cite{Fu80}] \label{wedgetheorem}Let $U$ be a finely open set in $\CC$ and $\alpha>1$. Then for every $w \in U$, there exists a fine neighborhood $V$ of $w$ such that any two distinct points $a, b \in V$ can be connected by a wedge contained in $U$ of total length less than $\alpha|a-b|$.
 \end{theorem}
 
 Lyons, see \cite[p. 16]{Lyons1980} had already proven that  $a$ and $b$ can be connected by a  polygonal path consisting of two line segments of total length less than $\alpha |a-b|$ and his proof essentially contained Theorem \ref{wedgetheorem}.
 For an even stronger result see Gardiner, \cite[Theorem A]{Gard}. 

We will also need the following elementary lemma. Let $B(x,r)$ denote the open disc with radius $r$ about $x\in \CC$, $\overline B(x,r)$ its closure, and let $C(x,r)$ denote its boundary. 
\begin{lemma} \label{Lem1} Let $V$ be a fine neighborhood of $a\in\CC$.  Then there exists $C_0>1$ such that for $C>C_0$ and every  $n\in \NN$ there exists $t\in [C^{-n-1}, C^{-n}]$ with  $C(a, t)\subset V$.
\end{lemma}
\begin{proof} We can assume $a=0$.  A  local basis at $0$ for the  fine topology consists of the sets $B(0,r, h)= \{z\in B(0,r) :  h(z)>0\}$,
where $r>0$, $h$ is subharmonic on $B(0,r)$ and  $h(0)=1$. See \cite[Lemma 3.1]{ElMarzguioui2006} for a proof. Thus for some $h$ and $r$ we have  $V\supset W:=\{z\in B(0,r): h(z)\ge 1/2 \}$ and  $B(0,r)\setminus V$ is contained in $F:=B(0,r)\setminus W$, which is Euclidean open, hence an $F_\sigma$, that is thin at $0$, because $W$ contains the finely open set $B(0,r,h)$ which contains $0$. 
Theorem 5.4.2 in \cite{Ran} states for given $r>0$ that there is a constant $\Gamma>0$ such that
\[\int_E \frac1x dx<\Gamma<\infty,\]
where 
\[E:=\{s:0<s<r\mid \ \exists \theta \text{\ with\ } s e^{i\theta}\in F\}.\]
Let  $C_0=\max\{e^{\Gamma}, 1/r\}$. For $C>C_0$ each interval of the form  $[C^{-n-1},C^{-n}]$ is contained in $(0,r)$, but can not be contained in $E$. Hence there exists $t\in [C^{-n-1},C^{-n}]$ such that $C(0,t)\cap F=\emptyset$, that is, $C(0,t)\subset W\subset V$.
\end{proof}

For  more  information on the fine topology see \cite[Part I, Chapter XI]{Doob},
\cite[Chapter VII]{ArGa}.
\medbreak

There are several equivalent definitions for finely holomorphic functions on a fine domain in $D\subset\CC$, cf.~\cite{Fu8081, Fu81, Fu88, Lyons1980}. 

\begin{definition} A function $f$ on a fine domain $D$ is called finely differentiable at $z_0\in D$ with fine complex derivative $f'(z_0)$ if there exists $f'(z_0)\in\CC$ such that for every $\eps>0$ there exists a fine neighborhood $V\subset D$ of  $z_0$ such that 
\[\left|\frac{f(z)-f(z_0)}{z-z_0}-f'(z_0)\right|<\eps \quad\text{for all } z\in V.\]
In other words, the limit $(f(z)-f(z_0))/(z-z_0)$ exists as $z-z_0\to 0$ finely.

The function $f$ is finely holomorphic on $D$ if and only if $f$ is finely differentiable  at every point of  $D$ and $f'$ is finely continuous  on $D$.
\end{definition}
We will use the following characterizations of fine holomorphy. 
\begin{theorem}[\cite{Fu81}] \label{finehol} Let $f$ be a complex valued function on a fine domain $D$. The following are equivalent
\begin{enumerate}
\item The function $f$ is finely holomorphic on $D$.
\item  Every point $z_0\in D$ admits a fine neighborhood  $V\subset D$ such that $f$ is a uniform limit of rational functions on $V$. 
\item The functions $f$ and $z\mapsto zf(z)$ are both (complex valued) finely harmonic functions on $D$.
\end{enumerate}

\end{theorem}

In the following theorem we collect properties of finely holomorphic functions that indicate how much this theory resembles classical function theory.

\begin{theorem}[\cite{Fu8081}] Let $f: D\to \CC$ be finely holomorphic. Then
\begin{enumerate}
\item  The function $f$ has fine derivatives $f^{(k)}$ of all orders $k$, and these are finely holomorphic on $D$.
\item  Every point $z_0\in D$ admits a fine neighborhood $V\subset D$ such that for every $m=0,1,2,\ldots$ 
\begin{equation} \label{Tay}
\left|f(w)-\sum_{k=0}^{m-1}\frac{f^{(k)}(z)}{k!}(w-z)^k\right|\big/ |w-z|^m
\end{equation} is bounded  on $V\times V$ for $z\ne w$.
\item At any point $z$ of $D$ the Taylor expansion  \eqref{Tay} uniquely determines $f$ on $D$. (If all coefficients equal 0, then $f$ is identically $0$.)
\end{enumerate}
\end{theorem}
We now introduce finely isolated singularities.
\begin{definition} Let $D$ be a fine  domain, $a\in D$, and $f$ finely holomorphic on $D\setminus \{a\}$.
\begin{itemize}
\item  If $f$ extends  as a finely holomorphic function  to all of  $D$, then $f$ has a removable singularity at $a$.
\item If $f(z)\to\infty$ for $z\to a$ finely, $f$ has a pole at  $a$.
\item If $f$ has no pole at $a$ nor a removable singularity, then $f$ has an essential singularity at $a$.
\end{itemize}
\end{definition}
\begin{theorem}\label{Riemann} Let $D$ be a fine domain and $a\in D$. Let $f$ be finely holomorphic on $D\setminus\{a\}$ and suppose that $f$ is bounded on a fine neighborhood $V$ of  $a$. Then $a$ is a removable singularity of $f$.
\end{theorem}
\begin{proof} (As indicated in \cite{Fu88}.) We apply Theorem \ref{finehol}, no. 3. Clearly  $f$ and $z\mapsto zf(z)$ are bounded, finely harmonic functions on  $V\setminus \{a\}$.  \cite[Corollary  9.15]{Fu} states that these functions extend to  be finely harmonic on all of $D$ and again by Theorem \ref{finehol} $f$ is finely holomorphic on $D$.
\end{proof}

\section{Fine domains of existence}\label{sec2}

For a set $A$ in $\CC$, we denote its fine interior by $\fint A$, its fine closure by $\fcl{A}$ and its fine boundary by $\partial_f A$.

\begin{definition}\label{Def.existence} A fine domain $U$ is called a \emph{fine domain of existence} if there exists a finely holomorphic function $f$ on $U$ with the property that for every fine domain $V$ that intersects $\partial_f U$ and for every fine component $\Omega$ of $V\cap U$, the restriction $f|_\Omega$ admits no finely holomorphic extension to $V$.
\end{definition}
\begin{definition} Let $U$ be a fine domain in $\CC$. If there exist compact sets $K_n\subset U$ such that $U = \bigcup_{n=1}^\infty K_n$ with $K_1 \subset \fint K_2 \subset K_2 \subset \fint K_3 \subset \cdots$, then the sequence $\{K_n\}$  is called a \emph{fine exhaustion} of $U$.  If, moreover, the $K_n$ have the property that every bounded component of $\CC \setminus K_n$ contains a point from $\CC \setminus U$, we call $\{K_n\}$ a \emph{special fine exhaustion} of $U$.
\end{definition}
We have the following lemma, which is a consequence of the Lusin--Menchov property of the fine topology.
\begin{lemma}[{\cite[Corollary 13.92]{MR2589994}}]\label{LuMe}
Let $U$ be a finely open set and let $K$ be a compact subset of $U$. Then there exists a Borel finely open set $V$ such that $K \subset V \subset \overline{V} \subset U$. 
\end{lemma}

\begin{prop}\label{generalUfinecompactexhaustion}
Let $U$ be a fine domain that is a Euclidean $F_\sigma$. Then $U$ admits a special fine exhaustion.
\end{prop}
\begin{proof}
Let $U = \bigcup_{n=1}^\infty F_n$ for compact sets $F_1 \subset F_2 \subset \cdots$. Fix a strictly increasing sequence $(r_n)_{n \geq 1}$ tending to infinity such that $F_n \subset B(0,r_n)$ for all $n \geq 1$. We put $K_1 := F_1$. For the construction of $K_2$, note that $F_2 \cup K_1 (=F_2)$ is a compact subset of $U$. By Lemma \ref{LuMe}, there exists a finely open set $V_2$ such that $F_2 \cup K_1 \subset V_2 \subset \overline{V_2} \subset U$. Then the set $K_2 := \overline{V_2} \cap \overline{B(0,r_2)}$ is compact and
\[
	\fint K_2 
    = \fint\overline{V_2} \cap \fint\overline{B(0,r_2)} \supset V_2 \cap B(0,r_2) \supset F_2 \supset F_1 = K_1.
\]
As induction hypothesis, suppose that for some $n \geq 2$, we have found compact sets $K_1, \ldots, K_n$ such that $K_j \subset \fint K_{j+1}$ for all $1 \leq j \leq n-1$, $F_j \subset K_j$ for all $1 \leq j \leq n$ and $K_j \subset B(0,r_{j+1})$ for all $1 \leq j \leq n$ and that we have found finely open sets $V_2, \ldots, V_n$ such that $F_j \cup K_{j-1} \subset V_j \subset \overline{V_j} \subset U$ for all $2 \leq j \leq n$. 

We now prove the induction step. Note that the set $F_{n+1} \cup K_n$ is a compact subset of $U$. By the previous lemma, there exists a finely open set $V_{n+1}$ such that $F_{n+1} \cup K_n \subset V_{n+1} \subset \overline{V_{n+1}} \subset U$. The set $K_{n+1} := \overline{V_{n+1}} \cap \overline{B(0,r_{n+1})}$ is compact, contained in $U$ and in $B(0,r_{n+2})$ and
\[
	\fint K_{n+1} = \fint\overline{V_{n+1}} \cap \fint\overline{B(0,r_{n+1})} \supset V_{n+1} \cap B(0,r_{n+1}) \supset F_{n+1} \cup K_n
\]
and therefore $\fint K_{n+1} \supset K_n$ and $K_{n+1} \supset F_{n+1}$. This proves the induction step.

Consequently, we can find compact subsets $K_1, K_2, \ldots$ of $U$ such that $K_1 \subset\fint K_2 \subset K_2 \subset\fint K_3 \subset \cdots$ and $F_n \subset K_n$ for all $n \geq 1$. It follows that $U = \bigcup_{n=1}^\infty K_n$, which proves that $U$ admits a fine exhaustion.

We will next adapt this fine exhaustion so that every bounded component of $\CC \setminus K_n$ contains a point from $\CC \setminus U$. 
Observe that $K_n=\CC\setminus\bigcup_j D^n_j$, where for each $n$ the $D^n_j$ are the countably, possibly infinitely, many mutually disjoint open components of $\CC\setminus K_n$. Let $D^n_0$ be the unbounded component.
Then for any $n$ and any finite or infinite sequence $1\le j_1\le j_2\le \cdots$ the set $K_n\cup \bigcup_k D^n_{j_k}$ is compact too. 
For every $n$ we set $K^*_n=K_n\cup\bigcup_k D^n_{j_k}$, where the $D^n_{j_k}$ are those components, if any, of $\CC\setminus K_n$ that are completely contained in $U$. We claim that $K_n^*\subset \fint(K_{n+1}^*)$.

Indeed, if $x\in K_n$ then $x\in \fint K_{n+1}\subset\fint (K_{n+1}^*)$. Now let $x\in D_{j_k}^n(\subset U)$.  For the proof that $x\in K^*_{n+1}$ we may suppose that $x\notin K_{n+1}$. Then $x$ belongs to a (necessarily bounded) component of $ \CC\setminus K_{n+1}$ that is completely contained in $D_{j_k}^n\subset U$, hence $x\in K^*_{n+1}$. It follows that  $D_{j_k}^n\subset K_{n+1}^*$ and as  $D_{j_k}^n$ is Euclidean open, $D_{j_k}^n\subset \fint(K_{n+1}^*)$. This proves the claim.
\end{proof}

\begin{figure}
 \centering
    \begin{subfigure}[b]{0.3\textwidth}
\begin{tikzpicture}[scale=0.02]

\clip (0,0) circle (100.5);

\fill (90:0) circle (0.1);

\draw (0,0) circle (100);
\draw (0,-100) circle (100);

\filldraw[white] (0,0) -- ([shift=(20:20)]0:0) arc (20:340:20) -- cycle;

\draw ([shift=(22.5:20)]90:0) arc (22.5:337.5:20);

\filldraw[fill=green!20,draw=green!50!black] (0,0) -- (10,10) -- (0,20) -- (-10,10) -- cycle;
\draw (0,0) -- (5,10) -- (0,20);

\filldraw[fill=green!20,draw=green!50!black] (0,0) -- (-10,-10) -- (0,-20) -- (10,-10) -- cycle;
\draw (0,0) -- (-7,-10) -- (0,-20);

\end{tikzpicture}

 \end{subfigure}
\qquad
    \begin{subfigure}[b]{0.3\textwidth}

\begin{tikzpicture}[scale=0.07]

\clip (-30,-30) rectangle (30,30);


\draw (0,0) circle (100);
\draw (0,-100) circle (100);

\filldraw[white] (0,0) -- ([shift=(20:20)]0:0) arc (20:340:20) -- cycle;

\draw ([shift=(22.5:20)]90:0) arc (22.5:337.5:20);

\filldraw[fill=green!20,draw=green!50!black] (0,0) -- (10,10) -- (0,20) -- (-10,10) node[above]{\tiny $S_1$} -- cycle;
\draw (0,0) -- (5,10) node[pos=1,left]{\tiny $W_1$} -- (0,20);

\filldraw[fill=green!20,draw=green!50!black] (0,0) -- (-10,-10) node[below]{\tiny $S_2$} -- (0,-20) -- (10,-10) -- cycle;
\draw (0,0) -- (-7,-10) node[pos=0.9,right]{\tiny $W_2$} -- (0,-20);

\filldraw[fill=red!20, draw=red!50, opacity=0.4] (-50,-7.6536686473) rectangle (50,7.6536686473);
\draw (-16,-3) node[above] {\tiny $b$} -- (40,12) node[pos=0.8,above] {\tiny $q$};
\filldraw (-16,-3) circle (0.1);

\end{tikzpicture}

 \end{subfigure}
 \qquad
    \begin{subfigure}[b]{0.3\textwidth}
\begin{tikzpicture}[scale=0.02]

\clip (0,0) circle (100.5);


\draw (0,0) circle (100);
\draw (0,-100) circle (100);

\filldraw[white] (0,0) -- ([shift=(20:20)]0:0) arc (20:340:20) -- cycle;

\draw ([shift=(22.5:20)]90:0) arc (22.5:337.5:20);

\draw (0,0) -- (3.82683432365,7.6536686473);

\draw (0,0) -- (-5.35756805311,-7.6536686473);

\end{tikzpicture}

 \end{subfigure}
\qquad
    \begin{subfigure}[b]{0.3\textwidth}

\begin{tikzpicture}[scale=0.08]

\clip (-30,-30) rectangle (30,30);

\draw (0,0) circle (100);
\draw (0,-100) circle (100);

\filldraw[white] (0,0) -- ([shift=(20:20)]0:0) arc (20:340:20) -- cycle;

\draw ([shift=(22.5:20)]90:0) arc (22.5:337.5:20);

\draw (0,0) -- (3.82683432365,7.6536686473);

\draw (0,0) -- (-5.35756805311,-7.6536686473);

\end{tikzpicture}
 \end{subfigure}

\caption{\label{boogsnijconstructie}Replacing a small circular arc.}
\end{figure}
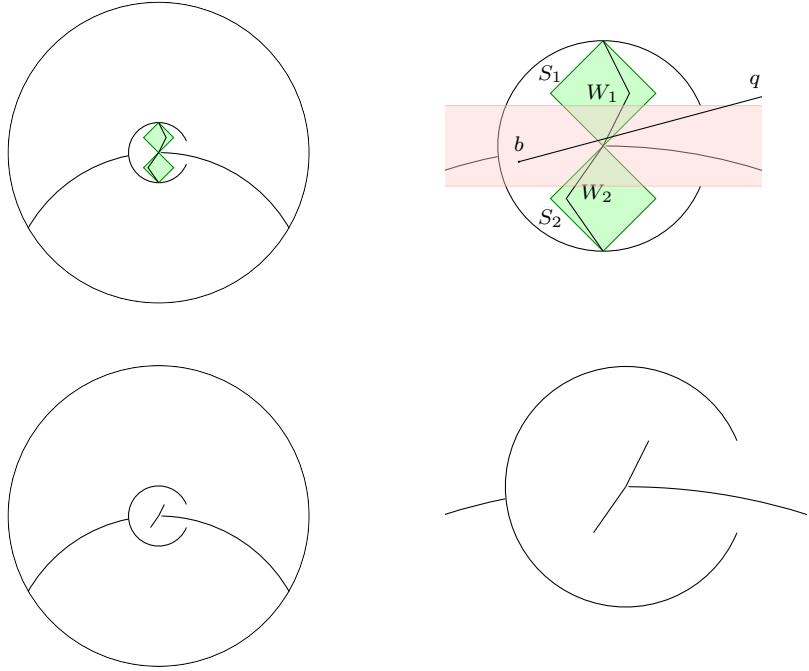

We will need the following lemma.  Figure \ref{boogsnijconstructie} illustrates its content.

\begin{lemma}\label{boogsnijlemma} Let $D$ be an open disc, $U$ a finely open subset of $D$, and let $a\in C\cap U$ where $C$ is an open  circular arc contained in $D$ with the property that $D\setminus C$ has two components.  

Then there exists a sequence $(r_j)_j$ of positive numbers decreasing to 0, such that for every $r_j$ there exists a compact set $C'(=C'_j)\subset U$, which is the union of four arcs, such that 
\begin{enumerate}
\item $D\setminus C'$ is connected;
\item $C'\setminus \overline{B(a,r_j)}=C\setminus \overline{B(a,r_j)}$;
\item Every wedge $\ell = [p,b] \cup [b,q]$ with $|p-b|=|b-q|>r_j$
that meets $C$ also meets $C'$.
\end{enumerate} 
\end{lemma}

\begin{proof} By Theorem \ref{wedgetheorem} there exists a fine neighborhood $V\subset U$ of $a$ such that any two distinct points $p,q\in V$ can be connected by a wedge $L\subset U$ of length less than $\sqrt 2 |p-q|$. Lemma \ref{Lem1} provides us with a sequence $(r_j)_j$ such that $C(a,r_j)\subset V$. 

After scaling and rotating we may assume that $C=C(-i,1)\cap D$ and that $a=0$.  
We can also assume all $r_j$ less than $1$ and so small that $\overline{B(0, r_j)}\subset D$. Now we fix $r=r_j$ and set
\[C_1=\left(C\setminus\left( B(0,r)\cap\{\Re z <0\}\right)\right)\cup\left(\{re^{i\theta}: \pi/6\le \theta\le 11\pi/6\}  \right).\]
$C_1$ is contained in $V$ along with $C(0,r)$, therefore $i r$, respectively $-i r$, (both on $C_1$) can be connected to $0(=a)\in V\subset U$ by a wedge $W_1$, respectively $W_2$, of length less than $\sqrt 2r$ contained inside $U$. Observe that the wedges $W_1, W_2$ are contained in $\overline{B(0,r)}$, and in fact $W_1$ is contained in the closed square $S_1$ with diagonal $[0,ir]$ and similarly $W_2$ is contained in the closed square $S_2$ with diagonal $[0,-ir]$, see Figure \ref{boogsnijconstructie}.

Let $\ell= [p,b] \cup [b,q]$ be any wedge with $|p-b|=|b-q|>2r$ that meets $C$ and assume that $\ell$ does not meet $C_1$. Then one of its constituting segments, say $[b,q]$, meets $C\cap B(0,r)\cap\{\Re z<0\}$ and because $|q-b|>2r$, it meets $C(0,r)$ in a point $re^{i\theta_0}$ with $-\pi/6<\theta_0<\pi/6$. Therefore, as $r<1$,  $[b,q]\cap(S_1\cup S_2)$ is contained in $\{z: |\Im z|<r\sin(\pi/6)=r/2\} \cap(S_1\cup S_2)$, as indicated in  Figure  \ref{boogsnijconstructie}.

Now we define $L_j=W_j\cap \{z: |\Im z|\le r/2
\}\subset U$, $j=1,2$ and $C'=C_1\cup L_1\cup L_2\subset U$ is the the union of three circular arcs and one arc consisting of two straight segments, all contained in $U$. Clearly $\ell$ meets $C'$ and $D\setminus C'$ is connected.
\end{proof}

\begin{prop}\label{prop2}
Let $U \subsetneq \CC$ be a nonempty fine domain that is a Euclidean $F_\sigma$ and let $\emptyset\ne F_1\subset F_2\subset \cdots$ be an increasing sequence of compact sets in $\CC\setminus U$.  
Then there exist sequences $(K_n)_{n \geq 1}$ and $(L_n)_{n \geq 1}$ of nonempty compact sets such that \begin{enumerate}
\label{alph}
\item[(a)] $K_1 \subset \fint K_2 \subset K_2 \subset\fint K_3 \subset \cdots$ and $\bigcup_{n=1}^\infty K_n=U$, 
\item[(b)] $d(K_n, L_n) > 0$ for all $n \geq 1$,
\item[(c)] $L_n \subset K_{n+1}$ for all $n \geq 1$,
\item[(d)] for any $n \geq 1$, every two points $w,z \in \CC \setminus (K_n \cup L_n)$ in the same component of $\CC \setminus K_n$ also lie in the same component of $\CC \setminus (K_n \cup L_n)$, 
\item[(e)] every bounded component of $\CC \setminus (K_n \cup L_n)$ contains a point from $\CC \setminus U$,
\item[(f)] every wedge of length at least $1/n$ that meets $F_n$ and $K_n$, also meets $L_n$.
\end{enumerate}
Moreover, there exists a sequence $(f_\nu)_{\nu \geq 1}$ of rational functions with poles in $\CC \setminus U$ such that 
\begin{enumerate}
\item the sup-norm $\|f_n - f_{n-1}\|_{K_n} < 1/2^n$ for all $n \geq 2$ and
\item $|f_n| > n$ on $L_n$ for all $n \geq 1$.
\end{enumerate}
The sequence $(f_\nu)_\nu$ converges uniformly on any $K_n$ to a function $f$ that is finely holomorphic on $U$ and has the property that $|f| > n-1/2^{n}$ on $L_n$ for any $n \geq 1$.
\end{prop}

\begin{proof} For $(K_n)$ we take a special fine exhaustion of $U$, $U=\bigcup_{n=1}^\infty K_n$, which exists in view of Proposition \ref{generalUfinecompactexhaustion}.
We next construct a sequence $(L_n)_{n \geq 1}$ of compact sets such that $d(L_n, K_n) > 0$ and $L_n \subset K_{n+1}$ for all $n \geq 1$.   

Let $n \geq 1$. The set $F_n\subset\CC\setminus U$ is compact and is therefore covered (uniquely) by those finitely many components of $\CC \setminus K_n$ that meet $F_n$. We denote these components by $D_1^n, D_2^n, \ldots, D_{k_n}^n$. Since $K_n$ and $F_n$ are compact and $K_n \cap F_n = \emptyset$, it follows that $\delta_n := d(K_n, F_n) > 0$. Consequently,
\begin{equation}\label{eq2.5}
\{z \in \CC : d(z, K_n) < \delta_n\} \cap F_n = \emptyset.
\end{equation}
Since $K_n \subset \fint K_{n+1}$, by Lemma \ref{Lem1} each $z \in \partial K_n$ is the center of a circle $C(z,\epsilon_z)$ contained in $\fint K_{n+1}$ of radius $\epsilon_z$ less than $\delta_n$. 
We will now construct $L_n$ by constructing the intersections $L_n \cap D_k^n$ (with some abuse of notation) in each of the components $D_1^n, D_2^n, \ldots, D_{k_n}^n$ and taking the union of these sets.

Fix a component $D_k^n$. Since $\partial D_k^n \subset \partial K_n$ is compact and is covered by the open disks $B(z,\epsilon_z)$ where $z \in \partial D_k^n$, it follows that there are finitely many points $z_1, \ldots, z_{m} \in \partial D_k^n$ such that $\partial D_k^n \subset \bigcup_{j=1}^{m} B(z_j, \epsilon_{z_j})$. Note that $(\partial \bigcup_{j=1}^{m} B(z_j, \epsilon_{z_j})) \cap D_k^n$ is a finite union of closed circular arcs $\mathcal{C}_l$ --and possibly finitely many points disjoint from these arcs-- of the components of $\bigcup_{j=1}^m\partial B(z_j, \epsilon_{z_j})\setminus\bigcup_{k=1
}^{m} B(z_k, \epsilon_{z_k})$. We define a compact set $L^*_{n,k}:=\bigcup_{l=1}^M \mathcal{C}_l\subset D_k^n$. Also note that every path $\gamma$, in particular every wedge, that connects a point $z \in F_n \cap D_k^n$ and a point $w \in K_n$ must intersect one of the arcs $\mathcal{C}_l$ in $L^*_{n,k}$, because $d(z,K_n)\ge\delta_n>\epsilon_z$ in view of (\ref{eq2.5}), and $w\in K_n\subset\mathbb C\setminus D^n_k$. The open set $\CC \setminus L_{n,k}^*$ has only finitely many, say $N\ge 2$, components as each of the $\mathcal{C}_l$ belongs to the boundary of two of these components. We will replace $L_{n,k}^* $ by $L_n\cap D_k^n$ in the following way to achieve that $\CC \setminus (L_n \cap D_k^n)$ is connected and the properties (a), (b) and (c) are kept.

To do this, note that $\mathcal{C}_1$  belongs to the boundary of two components of $\CC \setminus L_{n,k}^* $. Pick a point $a$ of $\mathcal{C}_1$ that is not an endpoint of this closed arc. Since $a\in \fint K_{n+1}\setminus K_n$, clearly $a \notin F_{n+1}\supset F_n$, and again since $a$ lies in the relative interior of the arc $\mathcal{C}_1$, we see that for all $r > 0$ small enough, $\overline{B (a,r)} \cap K_n = \emptyset = \overline{B(a,r)} \cap F_n$ and $\overline{B(a,r)}\cap L_{n,k}^*$ is contained in the relative interior of $\mathcal{C}_1$. 

We apply Lemma \ref{boogsnijlemma} with $D=B(a,r)$, $U=\fint K_{n+1}\cap B(a,r)$, $r_j<1/n$ and obtain a compact $\mathcal{C}_1'$. Then $\CC\setminus (\mathcal{C}_1'\cup\bigcup_{l=2}^M\mathcal{C}_l)$ has $N-1$ components. 
If $\mathcal{C}_2$ belongs to the boundary of two different components of $\CC\setminus (\mathcal{C}_1'\cup\bigcup_{l=2}^M\mathcal{C}_l)$ we replace it by $\mathcal{C}'_2$ likewise, reducing the number of components of the complement by one again. If $C_2$ belongs to the boundary of only one component, we just put $\mathcal{C}'_2= \mathcal{C}_2$. Proceeding in this way, we end with $L_n\cap D_k^n:=\bigcup_{l=1}^M \mathcal{C}'_l$, the complement of which is connected. 
By setting $L_n := \bigcup_{k=1}^n (L_n \cap D_k^n)$, we now find that any two points $w, z \in \CC \setminus (K_n \cup L_n)$ that are in the same component of $\CC \setminus K_n$ are in the same component of $\CC \setminus (K_n \cup L_n)$ as well, while Lemma \ref{boogsnijlemma} guarantees that (f) is satisfied. Also, note that $L_n \subset\fint K_{n+1} \setminus K_n$, so that $K_n \cap L_n = \emptyset$. It follows that $d(K_n, L_n) > 0$ for every $n \geq 1$.
We have obtained sequences $(K_n)_{n \geq 1}$ and $(L_n)_{n \geq 1}$ of compact sets that satisfy the properties (a) -- (f).

Because of property (b), we can
apply Runge's theorem recursively for each $n$ to the holomorphic function $g_n$ that equals $0$ on an open set containing $K_n$ and is equal to  $\sum_{j=1}^{n-1}\max_{z\in L_n}\{|R_j(z)|\}+n+1$ on an open set containing $L_n$. Thus there exists for each $n$ a rational function $R_n$ with 
\begin{align}
\label{r1} |R_n|&<1/2^n,\quad\text{on  $K_n$};\\
\label{r2}  |R_n|&>\sum_{j=1}^{n-1}\max_{z\in L_n}\{|R_j(z)|\}+n,\quad\text{on   $L_n$},
\end{align}
while $R_n$ has at most one pole in a preassigned point in each of the bounded components of $\mathbb C\setminus (K_n\cup L_n)$. Because of property (e) these poles may be taken from $\mathbb C\setminus U$. Set 
\[f_k= \sum_{j=1}^k R_j.	\]
Then $f_k$ clearly satisfies (1) and (2).
It now follows from (1) and (a) that there exists a function $f$ on $U$ such that $f_k \to f$ uniformly on any $K_n$. Note that $f$ is finely holomorphic on $U$, because if $z \in U$, then $z \in K_n$ for some $n \geq 1$ and, by property (a), $K_{n+1}$ is a compact fine neighborhood of $z$ on which $f_k \to f$ uniformly. Furthermore, it follows from (2) and (c) above that for any $n \geq 1$ and  $z\in L_n$ 
\[|f_n(z)| \geq |R_n(z)|-\sum_{j=1}^{n-1}|R_j(z)|> n.\] 
Finally, $|f|> n-1/2^n$ follows immediately from (1) and (2). 
\end{proof}

We shall prove in Theorem \ref{Ufinecompactexhaustion} that  under suitable conditions on the fine domain, the function that we have just constructed admits no finely holomorphic extension outside the domain. We need some facts about (regular) fine domains.

It is known that every regular fine domain $U$ is an $F_\sigma$ because $\mathbb C\setminus U$ is a base in the sense of Brelot, and hence a $G_\delta$. Moreover, we need the following Lemma, which is  (iv) of \cite[Theorem 7.3.11]{ArGa}.

\medskip\noindent
\begin{lemma} \label{lemma1.1}Every finely closed set $A\subset\mathbb C$ is the disjoint union of a (Euclidean) $F_\sigma$ and a polar set.\end{lemma}

\begin{lemma}\label{lem2.7} Let $\emptyset\ne U\subset V$ be fine domains and let $E$ be a polar subset of $\CC\setminus U$. Then either there exists a point $a\in (\partial_f U\cap V)\setminus E$, or $V=U\cup e$ for the polar set $e:=(V\setminus U)\cap E$, whereby $e$ is empty and hence $U=V$ in case $U$ is a regular fine domain.
\end{lemma}

\begin{proof} 
If $(\partial_fU\cap V)\setminus E$ is empty then $(\partial_fU\cap V)\subset E\cap(V\setminus U)=e$, hence the fine boundary $\partial_fU\cap(V\setminus e)$ of $U$ relative to $V\setminus e$ is empty. Since $V\setminus e$ is finely open and finely connected along with $V$, by \cite[Theorem 12.2]{Fu}, and since $U\ne\emptyset$ this means that $U=V\setminus e$, that is, $V=U\cup e$ as claimed. If $U$ is regular then $e$ is empty because $e$ and $U$ are disjoint and $V$ is finely open.
\end{proof}

\begin{theorem}\label{Ufinecompactexhaustion}
Let $U \subset \CC$ be a non-empty fine domain that is either (1) both Euclidean $F_\sigma$ and  Euclidean $G_\delta$, or (2) a regular fine domain. Then $U$ is a fine domain of existence.
\end{theorem}
\begin{proof}
\medbreak In both cases $U$ is an $F_\sigma$, and we can write  $U=\bigcup_{n=1}^\infty K_n$ with $K_j\subset K_{j+1}$ compact subsets of $U$. Because of Lemma \ref{lemma1.1}, $\CC \setminus U=F\cup E$, where $F=\bigcup_{n=1}^\infty F_n$ with $F_j\subset F_{j+1}$ compact in $\CC\setminus U$, is a Euclidean $F_\sigma$ and $E$ is a polar set which we can assume  to be empty in Case (1).    Let $L_j\subset K_{j+1}$ be compact sets and $f$ a finely holomorphic function on $U$ as constructed in Proposition \ref{prop2}. 

Let $V$ be a fine domain that meets $\partial_fU$ and let $\Omega$ be a fine component of $U\cap V$. Then $\partial_f\Omega\cap V\subset \partial_f U\cap V$, because if $x\in \partial_f\Omega\cap V\cap U$ there would be a finely connected fine neighborhood of $x$ contained in $V\cap U$, contradicting that $\Omega$ is a fine component of $V\cap U$.
Suppose that $f|_\Omega$ admits a finely holomorphic extension $\tilde{f}$ defined on $V$. 
By Lemma \ref{lem2.7} applied to $\Omega$ in place of $U$, there is a point $z_0\in(\partial_f\Omega\cap V)\setminus E$, whereby in Case (1) $E=\emptyset$ and hence $z_0\in(\partial_f\Omega\cap V)\subset\partial_f U\cap V\subset F$. In Case (2),  if $V=\Omega\cup e$, then $U\cup V=U\cup e$, which is impossible because $U$ is a regular domain. Hence $z_0\in(\partial_f\Omega\cap V)\setminus E=(\partial_fU\cap V)\setminus E\subset F$. 

Again by Proposition \ref{prop2} there is a compact fine neighborhood $V_0$ of $z_0$ in $V$ on which $\tilde{f}$ is the uniform limit of rational functions with poles off $V_0$. In particular, $\tilde{f}$ is continuous and bounded on $V_0$, say $|\tilde{f}| \leq M$ on $V_0$. Choose a fine domain $V_1\subset V_0$ containing $z_0$ and such that, for given $\alpha>1$, any two distinct points $a,b$ of $V_1$ can be joined by a wedge of total length less than $\alpha|a-b|$ contained in $V_0$, see Theorem \ref{wedgetheorem}.

Choose $z_1\in\Omega\setminus\{z_0\}$. By the above and by Lemma \ref{Lem1}, there exists for every $\epsilon>0$ an $0<r_0 <\epsilon$ with $r_0<|z_1-z_0|$ such that $C(z_0, r_0) \subset V_1$. 
Choose $z_2\in\Omega$ with $|z_2-z_0|<r_0$ (possible since $z_0\in\partial_f\Omega\subset\partial\Omega$). There is a path $\gamma\subset\Omega$ joining $z_1$ and $z_2$ and meeting $C(z_0,r_0)$ at a point $w_0$ of $\Omega\cap V_1$. Thus $C(z_0,r_0)\cap\Omega\ne\emptyset$. Recall that also $z_0\in V_1$.
Let $l_0 = [w_0, b_0] \cup [b_0, z_0]$ with $|w_0 - b_0| = |b_0 - z_0|$ be a wedge of total length less than $\alpha|w_0 - z_0|$ contained in $V_0$. Since $w_0 \in U$, we have $w_0 \in K_n$ for all $n$ large enough. Also, $z_0 \in F$, hence $z_0 \in F_n$ for all $n$ large enough. Therefore we can take an $n \geq 1$ large enough such that $w_0 \in K_n$ and $z_0 \in F_n$ and such that $n-1/2^{n-1}>M$  and $4/n < |w_0 - b_0| + |b_0 - z_0|$ hold simultaneously. Then by Proposition \ref{prop2} (f), $l_0$ meets $L_n\cap\Omega$.

Let $z\in l_0 \cap L_n\cap\Omega$. Since $z \in l_0 \cap\Omega\subset V_0$, we have $|\tilde{f}(z)| \leq M$. On the other hand, since $z \in L_n \cap\Omega$, we have $\tilde{f}(z) = f(z)$ and by (2) in Proposition \ref{prop2} $|\tilde{f}(z)| = |f(z)| \geq n-1/2^{n-1} > M$, which is a contradiction. We conclude that $f$ does not admit a finely holomorphic extension and that $U$ is a fine domain of existence.
\end{proof}

In the rest of this section we  extend the above theorem a little.

\begin{lemma}
Let $U_1$ and $U_2$ be fine domains of existence, and suppose that $(\partial_fU_1)\cap(\partial_fU_2)=\emptyset$. Then every fine component of $U_1\cap U_2$ is a fine domain of existence.
\end{lemma}
\begin{proof}
The assertion amounts to $U_1\cap U_2$ being a ``finely open set of existence'' (if nonvoid) in the obvious sense. Because $(\partial_fU_1)\cap(\partial_fU_2)=\emptyset$ we have, writing $U_1\cap U_2=U_0$,
\[\partial_f U_0=(U_1\cap\partial_fU_2)\cup(U_2\cap\partial_fU_1).\] 
By hypothesis there exists for $i=1,2$ a finely holomorphic function $h_i$ on $U_i$ such that for any fine domain $V$ that intersects $\partial_f U_i$ and any fine component $\Omega_i$ of $U_i\cap V$, $h_i|\Omega_i$ does not extend finely holomorphically to $V$. 

Let $V$ be any fine domain that intersects $\partial_fU_0$ and let $\Omega$ be a fine component of $V\cap U_0$. Without loss of generality we may assume that $V$ intersects  $U_2\cap\partial_fU_1$, and by shrinking $V$ that $V\subset U_2$. Then $\Omega$ is a fine component of $U_1\cap V=U_1\cap(U_2\cap V)=U_0\cap V$.  

The function $h_0:=h_1|U_0+h_2|U_0$ is then finely holomorphic on $U_0$. Because $h_2$ is finely holomorphic on $V\subset U_2$, the function $h_0|_\Omega$ is extendible to $V$ 
if and only if $h_1|_\Omega$ is extendible over $V$, 
and that is not the case. Therefore $U_0$ is a finely open set of existence.
\end{proof}
\begin{theorem}\label{Thm2}
Every fine domain $U\subset\mathbb C$ such that the set $I$ of irregular fine boundary points for $U$ is both an $F_\sigma$ and a $G_\delta$ is a fine domain of existence.
\end{theorem}
\begin{proof}
In the above lemma take $U_1=U_r$ (the regularization of $U$) and $U_2=\CC\setminus I$. 
Since $I$ is polar, $U_2$ is a fine domain along with  $\mathbb C$. Since $U$ is finely connected so is $U_r$. In fact, $U_r=U\cup E$, where $E$ denotes the polar set of finely isolated points of $\mathbb C\setminus U$. If $U_r=V_1\cup V_2$ with $V_1,V_2$ finely open and disjoint then
\[U=U_r\setminus E=(V_1\setminus E)\cup(V_2\setminus E)\]
with $V_1\setminus E$ and $V_2\setminus E$ finely open and disjoint. It follows that for example $V_1\setminus E=\emptyset$ and hence $V_1=\emptyset$, showing that $U_r$ indeed is finely connected.
By hypothesis, the complement of $U_2$ and hence $U_2$ itself is an $F_\sigma$ and a $G_\delta$.  Then by Theorem \ref{Ufinecompactexhaustion} $U_1$ and $U_2$ are  fine domains of existence. By Theorem \ref{Ufinecompactexhaustion}, $U_1\cap U_2=U_r\setminus I=U$ is indeed a fine domain of existence.
\end{proof}
\section{Fine domains that are not domains of existence}\label{sec3}

\begin{prop}\label{prop1} Let $(a_n)$ be a sequence in $\CC$ that converges to $a$. Suppose that $V_n$ is a fine neighborhood of $a_n$ of the form \[V_n=V_n(a_n,h_n,r)=\{z\in B(a_n,r): h_n(z)>0\},\]
where $h_n$ is a subharmonic function on $B(a_n,r)$ such that $h_n(a_n)=1/2$ and $h_n<1$ on $B(a_n,r)$.
Then $\bigcup_{n=1}^\infty V_n$ is a (possibly deleted) fine neighborhood of $a$.
\end{prop}
\begin{proof} Let $r_1<r$. Then there exists $n_0>0$ such that for $n\ge n_0$ the function $h_n$ is defined on $B(a,r_1)$. Let
\[h=(\sup_{n\ge n_0}\{h_n|_{B(a, r_1)}: n\ge n_0\}),\]
on $B(a, r_1)$ and  let $h^*$ denote its upper semi-continuous regularization. Then $h^*$ is subharmonic, $h^*\le 1$, and $h^*(a)\ge\limsup_{n\to \infty}h(a_n)\ge 1/2$.
Let $V_0=\{z\in B(a,r_1): h^*(z)>0\}$. Then $V_0$ is a fine neighborhood of $a$, since $h^*$ is finely continuous. Because the set $X=\{h<h^*\}\setminus\{a\}$ is polar, it is finely closed, therefore the set $V=V_0\setminus X$ is a fine neighborhood of $a$.

We claim that $V\setminus\{a\}\subset\bigcup_{n\ge n_0} V_n$. Indeed, if $z\in V\setminus \{a\}$ then $h(z)=h^*(z)>0$ hence $h_n(z)>0$ for some $n\ge n_0$, and $z\in V_n$.
\end{proof}

\begin{prop}\label{mainprop} Let $E$ be a non-empty 
polar set in $\CC$ and suppose that $E$ is of the first Baire category in its Euclidean closure $K$. Suppose that $f$ is finely holomorphic on a finely open set $V$ such that $K\setminus E\subset V$.  Then there exist a Euclidean open ball $B$ that meets $K$,  and a finely open fine neighborhood $V_1$ of $K\cap B$ such that $f$ is bounded on $V_1\setminus E$.
\end{prop}

\begin{proof} Let $x\in V\setminus E$. Denote by $U_x$ a finely open subset of $V$ containing $x$ and having the property stated in Theorem \ref{wedgetheorem} applied to $V$. By shrinking we may arrange that $U_x$ has the form
\[U_x=U_x(x,h_x,r_x)=\{z\in B(x,r_x): h_x(z)>0\},\]
where $h_x$ is a subharmonic function on $B(x,r_x)$ such that $h_x(x)=1/2$ and $h_x<1$ on $B(x, r_x)$.

Let $X_j$ be the set of $x\in K\setminus E$ such that $r_x>1/j$, and $|f|\le j$ on $U_x$ and that $U_x$ satisfies Lemma \ref{Lem1} with $C_0=j$.
Then $K\setminus E=\bigcup_j X_j$, hence by the Baire category theorem, there exist $j_0$ and an open Euclidean ball which we may assume to be the unit disc $\DD$, such that $\emptyset\ne K \cap \DD\subset \overline X_{j_0}$ (the Euclidean closure of $X_{j_0}$).

\smallskip 
Let $w\in K\cap \DD$. Then $w=\lim_{n\to\infty} x_n$ for a sequence $(x_n)_n$ in $X_{j_0}$.  As $r_{x_n}>1/j_0$, Proposition \ref{prop1} gives us that  $\bigcup_{n}U_{x_n}$  is a (possibly deleted) fine  neighborhood of  $w$ on which $|f|\le j_0$ and $W_w:=\{w\}\cup\bigcup_{n}U_{x_n}$ is a fine neighborhood of $w$. If $w\in (K\setminus E)\cap \DD$, $|f|\le j_0$ on $W_w$. In fact, $f$ is finely holomorphic on $V$, hence $|f|$ is finely continuous on the finely open set $V\cap W_w$ which contains $w$ because $V\supset K\setminus E$. If $|f(w)|>j_0$ then $|f|>j_0$ on some fine neighborhood $Z$ of $w$, which contradicts that $Z$ meets $W_w\setminus\{w\}$ because $\mathbb C$ has no finely isolated points.

The set $V_1=\bigcup_{w\in K\cap \DD}W_w$ is a finely open fine neighborhood of $K\cap\DD$  with the property that on $|f|\le j_0$ on $V_1\setminus E$.
\end{proof}

\begin{remark} The reader should be aware that the condition \emph{$E$ is of the first Baire category in its Euclidean closure $K$} prevents sets like $E=\{1/n,n=1,2,\ldots \}$. Indeed, $E$ can not be written as countable union of nowhere dense subsets in its Euclidean closure $K=E\cup\{0\}$, because this set $E$ is relatively open in $K$. 
\end{remark}

\begin{theorem}\label{MT} Suppose that $E$ and $K$ are as in Proposition \ref{mainprop}, and that $V$ is a fine domain such that $V\cap E=\emptyset$ and $K\setminus E\subset V$. Then (1) $E\subset \partial_fV$ and (2)  $V$ is not a fine domain of existence.
\end{theorem}

\begin{proof} 

For (1) observe that $\partial_f V=\partial V$, because $V$ is connected and therefore not thin at any of its Euclidean boundary points.
Thus it suffices to show that $E\subset \partial V$. We have $E=\bigcup_{n=1}^\infty F_n$, a countable union of nowhere dense subsets of $K$, which is of the second category in itself. Therefore, $E$ is contained in the closure of $K\setminus E$. If not, suppose to reach a contradiction, that $O$ were an open set in $K$ not meeting $K\setminus E$, then $O= \bigcup_{n=1}^\infty (F_n\cap O)$, a countable union of nowhere dense sets, contradicting that $K$ is of the second category. As $K\setminus E$ is contained in the domain $V$, we find that $E$ must be contained in the Euclidean closure of $V$ and hence in $\partial_fV=\partial V$ because  $V\cap E=\emptyset$.

For (2) let $f$ be finely holomorphic on $V$. 
By Proposition \ref{mainprop} there exist an open ball $B$ that meets $K$ (and hence $E$) and a finely open set $V_1$ containing $K\cap B$ such that $f$ is bounded on  $V_1\cap V\cap B\subset V_1\setminus E$.  
Because $E$ is polar, $V_1\cap V$ is a deleted finely open fine neighborhood of every point in $E\cap B$, hence each point $a$ of $E\cap B$ is a finely isolated singularity of $f$. Theorem \ref{Riemann} implies that  $f$ extends over each $a\in E\cap B$  and hence has a finely holomorphic continuation  to the finely open set $V\cup[(V\cup E)\cap B]=V\cup(E\cap B)$, which contains $V$ properly because $(E\cap B)\subset \partial_fV$.
\end{proof}

\begin{example}\label{exa3.5} Let $V=\CC\setminus \QQ$, $E=\QQ$, and $K=\RR$. Then $V$, $E$, and $K$ satisfy the conditions of Theorem \ref{MT}, therefore $\CC\setminus \QQ$ is not a domain of existence for finely holomorphic functions. 

Similarly  $\CC\setminus (\QQ\oplus i\QQ)$ is not a domain of existence for finely holomorphic functions.  

\end{example}

\begin{question}From  Theorem \ref{MT} and Example \ref{exa3.5} it is clear that the requirement that $U$ be an $F_\sigma$ in Theorem  \ref{Ufinecompactexhaustion} can not be dropped. 
But is it perhaps true that every  fine domain that is a Euclidean  $F_\sigma$ is a fine domain of existence? It would suffice to prove this for $U=\mathbb C\setminus I$ with $I$ a polar $G_\delta$, see the proof of Theorem \ref{Thm2}.
\end{question}

\bibliographystyle{amsplain}

\begin{thebibliography}{Fu81a}
\bibitem[ArGa]{ArGa} Armitage, D. H. and Gardiner, S. J. \textit{Classical potential theory}, Springer Monographs in Mathematics. Springer-Verlag London, Ltd., London, 2001. 
\bibitem[Bo94]{Bo1892} Borel, \'E. \textit{Sur quelques points de la th\'eorie des fonctions},  Th\`ese, Paris 1894. Annales Scientifiques de l'\'E.N.S. \textbf{12}, (1895),  9--55.
\bibitem[Bo17]{Bo1917} Borel, \'E. \textit{ Le\c cons sur les fonctions monog\`enes uniformes d'une variable complexe}, Gauthier Villars, Paris, 1917.
\bibitem[Cho] {Choq} G.\ Choquet, \textit{Sur les points d'effilement d'un ensemble}, Ann. Inst. Fourier \textbf{9} (1959), 91--101.
\bibitem[Do]{Doob} Doob, J.L. \textit{Classical Potential Theory and Its Probabilistic Counterpart}, Grundl. math. Wiss. \textbf{262}, Springer, 1984.
\bibitem[Edl]{Edl}{Edlund, T.} \textit{Complete pluripolar curves and graphs},
   {Ann. Polon. Math.}
  \textbf{84}, (2004) {1},
     75--86.
\bibitem[EEW]{EEW} Edigarian, A., El Marzguioui, S., Wiegerinck, J. \textit{ The image of a finely holomorphic map is pluripolar}, Ann. Polon. Math. \textbf{97}, (2010), no. 2, 137--149. 
\bibitem[EMW]{ElMarzguioui2006}
El Marzguioui, S.
and Wiegerinck, J.\textit{
The Pluri-Fine Topology is Locally Connected},
Potential Anal.
\textbf{25}, (2006) 283--288.

\bibitem[Fu72]{Fu}Fuglede, B. \textit{Finely Harmonic Functions} Springer LNM. \textbf{289},  Springer, Berlin, 1972.
\bibitem[Fu80]{Fu80}Fuglede, B. \textit{ Asymptotic paths for subharmonic functions and polygonal connectedness of fine domains},
Seminar on Potential Theory, Paris, No. 5, pp. 97–-116, 
LNM. \textbf{814}, Springer, Berlin, 1980. 

\bibitem[Fu81]{Fu8081} Fuglede,  B. \textit{Sur les fonctions finement holomorphes.}  Ann. Inst. Fourier (Grenoble) \textbf{31}, (1981), no. 4, vii, 57–-88. 
\bibitem[Fu81a]{Fu81}Fuglede, B. \textit{ Fine topology and finely holomorphic functions.} 18th Scandinavian Congress of Mathematicians (Aarhus, 1980), pp. 22--38, Progr. Math. \textbf{11}, Birkh\"auser, Boston, Mass., 1981. 
\bibitem[Fu88]{Fu88} Fuglede, B. \textit{Finely holomorphic functions. A survey}, Rev. Roumaine Math. Pures Appl. \textbf{33}, (1988), no. 4, 283–-295. 
\bibitem[Gar]{Gard}Gardiner, S.J. \textit{Finely continuously differentiable functions}, Math. Z. \textbf{266}, (2010), no. 4, 851–-861.
\bibitem[LMN]{MR2589994} Luke{\v{s}}, J., Mal{\'y}, J., Netuka, I., Spurn{\'y}, J. \textit{Integral representation theory}, {Walter de Gruyter \& Co., Berlin}
    {de Gruyter Studies in Mathematics}
    \textbf{35}, {2010}.
\bibitem[Lyo]{Lyons1980}Lyons, T. J. \textit{Finely holomorphic functions}. J. Funct. Anal. \textbf{37}, (1980), no. 1, 1-–18. 
\bibitem[Pyr]{pyrih} Pyrih, P. \textit{Finely holomorphic functions and quasi-analytic classes},
  {Potential Anal.} \textbf{3}, (1994), {3}, {273--281}.
\bibitem[Ran]{Ran} Ransford, T. \textit{Potential Theory in the Complex Plane} London Math.~Soc. Student Texts \textbf{28}, 1995.

\end{thebibliography}

\end{document}